\documentclass[12pt,letterpaper]{amsart}
\usepackage{geometry,mathpazo}
\usepackage{amssymb,setspace,xcolor}
\usepackage[pdftex,breaklinks,colorlinks,
citecolor=blue,
urlcolor=blue,
linkcolor=blue]{hyperref}
\newtheorem{theorem}{Theorem}

\newtheorem{lemma}{Lemma}
\newtheorem{corollary}{Corollary}
\theoremstyle{remark}
\newtheorem{remark}{Remark}
\newcommand{\C}{\mathbb{C}}
\newcommand{\zb}{\overline{z}}
\newcommand{\D}{\Omega}

\newcommand{\ep}{\varepsilon}
\newcommand{\Dc}{\overline{\Omega}}
\newcommand{\dbar}{\overline{\partial}}
\newcommand{\sumprime}{\sideset{}{'}\sum}

\DeclareMathOperator{\ind}{ind}
\onehalfspace

\title{A Sufficient condition for compactness of Hankel operators}
\author{Mehmet \c{C}el\.ik}
\address[Mehmet \c{C}elik]{Texas A\&M University-Commerce, 
	Department of Mathematics, Commerce, TX 75429, USA}
\email{mehmet.celik@tamuc.edu}

\author{S\"{o}nmez \c{S}ahuto\u{g}lu}
\address[S\"{o}nmez \c{S}ahuto\u{g}lu]{University of Toledo, 
	Department of Mathematics \& Statistics, Toledo, OH 43606, USA}
\email{sonmez.sahutoglu@utoledo.edu}

\author{Emil J. Straube}
\address[Emil J. Straube]{Texas A\&M University, Department of Mathematics, 
	College Station, TX, 77843, USA}
\email{straube@math.tamu.edu}

\date{\today}
\begin{document}

\begin{abstract}
Let $\Omega$ be a bounded convex domain in $\mathbb{C}^{n}$. We show that if 
$\varphi \in C^{1}(\overline{\Omega})$ is holomorphic along analytic varieties in 
$b\Omega$, then $H^{q}_{\varphi}$, the Hankel operator with symbol $\varphi$, 
is compact. We have shown the converse earlier (\cite{CelikSahutogluStraube20b}), 
so that we obtain a characterization of compactness of these operators in terms of 
the behavior of the symbol relative to analytic structure in the boundary. A corollary 
is that Toeplitz operators with these nonvanishing symbols are Fredholm 
(of index zero).
\end{abstract}

\maketitle

\section{Introduction}\label{intro}

This paper is a follow up to our recent \cite{CelikSahutogluStraube20b}. There we 
proved the following necessary condition for compactness of a Hankel operator on 
a bounded convex domain $\Omega$ in $\mathbb{C}^{n}$ 
(\cite[Theorem 1]{CelikSahutogluStraube20b}). Let $n\geq 2$, $0\leq q\leq (n-1)$ and 
suppose $\varphi\in C(\overline{\Omega})$. If the Hankel operator $H^{q}_{\varphi}$ 
on $(0,q)$--forms is compact, then the symbol $\varphi$ is holomorphic along complex 
varieties of dimension $(q+1)$ (and higher) in the boundary (such varieties are actually 
affine and given by the intersection of $\overline{\Omega}$ with affine subspaces, 
see the discussion in \cite{CelikSahutogluStraube20b}). We also proved a rather partial 
converse: if the boundary contains only `finitely many' varieties, then the converse holds 
(\cite[Theorem 2]{CelikSahutogluStraube20b}). That is, in this case, if the symbol is 
holomorphic along the varieties, then the Hankel operator is compact. We conjectured 
that this restriction was an artifact of the proof. This is indeed the case: the main purpose 
of this paper is to prove the full converse (via a different method), with one small caveat: 
we need the symbol to be in $C^{1}(\overline{\Omega})$.

We first recall some notation and terminology from \cite{CelikSahutogluStraube20b}. 
The first section of that paper also contains an introduction to the subject and a discussion 
of previous work to which we refer the reader. For a bounded domain $\Omega$, we denote 
by $K^{2}_{(0,q)}(\Omega)$ the space of square integrable $\overline{\partial}$--closed 
$(0,q)$--forms on $\Omega$; its subspace of forms with holomorphic coefficients is 
denoted by $A^{2}_{(0,q)}(\Omega)$. The operator $P_{q}$ denotes the Bergman 
projection on $(0,q)$--forms, i.e. the orthogonal projection 
$P_{q}: L^{2}_{(0,q)}(\Omega) \rightarrow K^{2}_{(0,q)}(\Omega)$. For a symbol 
$\varphi\in L^{\infty}(\Omega)$, the associated Hankel operator 
$H^{q}_{\varphi}: K^{2}_{(0,q)}(\Omega) \rightarrow L^{2}_{(0,q)}(\Omega)$ is defined as 
\begin{equation}\label{definition}
H^{q}_{\varphi}f = \varphi f - P_{q}(\varphi f)
\end{equation}
for $f\in K^{2}_{(0,q)}(\Omega)$. 

\begin{theorem}\label{THEOREM1}
Let $\D$ be a bounded convex domain in $\C^n, \varphi\in C^1(\Dc)$, and 
$0\leq q\leq n-1$. Assume that $\varphi$ is holomorphic along the varieties in 
the boundary of dimension $q+1$ (and higher). Then $H^q_{\varphi}$ is compact 
on $K^2_{(0,q)}(\D)$.  
\end{theorem}
Combining Theorem \ref{THEOREM1} with \cite[Theorem 1]{CelikSahutogluStraube20b}, 
we have a complete characterization of compactness of $H^{q}_{\varphi}$ for symbols 
in $C^{1}(\overline{\Omega})$.
\begin{corollary}\label{Cor1}
Let $\D$ be a bounded convex domain in $\C^n, \varphi\in C^1(\Dc)$, and 
$0\leq q\leq n-1$. Then $H^{q}_{\varphi}$ is compact if and only if $\varphi$ is holomorphic 
along complex varieties of dimension $(q+1)$ (and higher) in the boundary.
\end{corollary}

\begin{remark}
Theorem 1 in \cite{CelikSahutogluStraube20b} only assumes that the Hankel operator is 
compact on $A^{2}_{(0,q)}(\Omega)$. Thus, on a bounded convex domain, if the symbol 
$\varphi$ is in $C^{1}(\overline{\Omega})$, then $H^{q}_{\varphi}$ is compact on 
$K^{2}_{(0,q)}(\Omega)$ if and only if it is compact on $A^{2}_{(0,q)}(\Omega)$. The 
analogous fact holds for the $\overline{\partial}$--Neumann operator, see \cite{FuStraube98}. 
For related questions, see \cite[Theorem 12.12]{HaslingerBook}.
\end{remark}

\begin{remark}
The converse of Theorem \ref{THEOREM1}, \cite[Theorem 1]{CelikSahutogluStraube20b}, 
holds for $\varphi$ only assumed in $C(\overline{\Omega})$. We do not know whether the 
conclusion of Theorem \ref{THEOREM1} still holds under this weaker assumption on $\varphi$. 
One natural approach would be to approximate $\varphi\in C(\overline{\Omega})$ by suitable 
symbols in $C^{1}(\overline{\Omega})$. Our proof below indicates that it suffices to 
approximate $\varphi\in C(\overline{\Omega})$, holomorphic along the varieties in the 
boundary, uniformly on $\overline{\Omega}$ to within $\varepsilon$ by 
$\varphi_{\varepsilon}\in C^{1}(\overline{\Omega})$ with 
$|\overline{\partial}_{V}\varphi_{\varepsilon}| \leq \eta(\varepsilon)$, 
where $\lim_{\varepsilon\rightarrow 0^{+}}\eta(\varepsilon)=0$ (rather than with 
$\overline{\partial}_{V}\varphi_{\varepsilon} = 0$). Here, $\overline{\partial}_{V}$ 
denotes $\overline{\partial}_{V}$ along $V\subset b\Omega$.
\end{remark} 

One of the reasons compactness of Hankel operators is of interest in operator theory is the 
connection to the Fredholm theory of Toeplitz operators. For $\varphi\in L^{\infty}(\Omega)$, 
the Toeplitz operator $T^{q}_{\varphi}$ is the bounded operator on $K^{2}_{(0,q)}(\Omega)$ 
defined by $T^{q}_{\varphi}f = P_{q}(\varphi f)$ (see e.g.\cite{FollandKohnBook}, 
\cite[Section IV.4]{Venugopal}). This connection provides the following consequence of 
Theorem \ref{THEOREM1}, which gives the Fredholm property for certain Toeplitz operators 
in the absence of compactness in the $\overline{\partial}$--Neumann problem.

\begin{corollary}\label{Fredholm}
 Let $\D$ be a bounded convex domain in $\C^n, \varphi\in C^1(\Dc)$, and 
$0\leq q\leq n-1$. Assume that $\varphi$ is holomorphic along the varieties in 
the boundary of dimension $q+1$ (and higher), and that $\varphi$ does not vanish 
on $b\Omega$ ($q=0$), or $\varphi$ does not vanish on $\overline{\Omega}$ ($1\leq q\leq (n-1)$). 
Then the Toeplitz operator $T^{q}_{\varphi}$ is Fredholm of index zero.
\end{corollary}
We provide the details of the argument in Section \ref{FredholmProof}, but note here 
that as usual, index zero results form the fact that the boundary of a convex domain 
(for $q=0$) and the domain itself (for $1\leq q\leq(n-1)$) are simply connected. The 
distinction between $q=0$ and $q>0$ arises because the restriction to a compact 
subset of $\Omega$ is a compact operator on $K^{2}_{(0,q)}(\Omega)$ when $q=0$ 
(i.e. on holomorphic functions), while it is not when $q>0$. Alternatively, 
one could consider Toeplitz operators on $A^{2}_{(0,q)}(\Omega)$, or on 
$\ker(\overline{\partial})\cap \text{dom}(\overline{\partial}^{*})\subset K^{2}_{(0,q)}(\Omega)$, 
endowed with the graph norm, to avoid this issue.

Our proof of Theorem \ref{THEOREM1}, given in Section \ref{proof}, owes much to 
\cite{Zimmer21}. Zimmer's use of Frankel's work (\cite{Frankel91}), and in particular 
his realization that the canonical potential for the Bergman metric on a convex domain 
has self bounded (complex) gradient, are important ingredients in this proof. The other two main 
ingredients come from \cite{StraubeICM, StraubeBook}, specifically from the proof given 
there of McNeal's result (\cite{McNeal02}) that property$(\widetilde{P_{q}})$ implies 
compactness of the $\overline{\partial}$--Neumann operator on $(0,q)$--forms 
(proof of Theorem 2.1 in \cite{StraubeICM}, proof of Theorem 4.29 
in \cite{StraubeBook}). The work in \cite{Zimmer21} also suggests possible 
generalizations of Theorem \ref{THEOREM1} (as well as of the main result in 
\cite{CelikSahutogluStraube20b}), but we do not pursue this direction here.

\section{Proof of Theorem \ref{THEOREM1}}\label{proof}

We first explain the strategy of the proof. Compactness of $H^{q}_{\varphi}$ will be 
established by showing that we have (a family of) compactness estimates: for all 
$\varepsilon >0$ there exists a compact operator $K_{\varepsilon}$ from 
$K^{2}_{(0,q)}(\Omega)$ to some Hilbert space $Z_{\varepsilon}$, such that 
\begin{equation}\label{compest}
\|H^{q}_{\varphi}\|^{2} \leq \varepsilon \|f\|^{2} + \|K_{\varepsilon}f\|^{2}
 \end{equation}
(see \cite[Lemma 4.3, part (ii)]{StraubeBook}). Observe that if $\varphi\in C^1(\overline{\Omega})$ 
and $f\in K^{2}_{(0,q)}(\Omega)$, then Kohn's formula 
$P_{q} = Id - \overline{\partial}^{*}N_{q+1}\overline{\partial}$ gives
\begin{equation}\label{norm2}
\|H^{q}_{\varphi}f\|^{2} 
= \langle\overline{\partial}^{*}N_{q+1}(\overline{\partial}\varphi\wedge f), 
\overline{\partial}^{*}N_{q+1}(\overline{\partial}\varphi\wedge f)\rangle 
= \langle N_{q+1}(\overline{\partial}\varphi\wedge f), \overline{\partial}\varphi\wedge f\rangle;
\end{equation}
here, $N_{q+1}$ denotes the $\overline{\partial}$--Neumann operator on $(0,q+1)$--forms, 
and $\langle\cdot,\cdot\rangle$ denotes the inner product in $L^{2}_{(0,q)}(\Omega)$ 
and $L^{2}_{(0,q+1)}(\Omega)$, respectively. The goal is now to establish the estimates 
\eqref{compest} for the right hand side of \eqref{norm2}.

Let us assume for this paragraph that $q=0$, to simplify the discussion. If $N_{1}$ were compact, 
i.e. if there were no varieties in the boundary, compactness of $H^{0}_{\varphi}$ would be 
immediate from \eqref{norm2}. Compactness of $N_{1}$ in the absence of varieties is established 
by showing that there exist families of plurisubharmonic functions, uniformly bounded  or 
with gradient uniformly self bounded (see below), that have Hessians as large as we please 
(property$(P)$ or property$(\widetilde{P})$, respectively). The large Hessians are used to 
produce compactness estimates. An analytic disc $V$ in the boundary is an obstruction to the 
existence of these families. On the other hand, in the right hand side of \eqref{norm2}, only 
the component of $N_{1}(f\overline{\partial}\varphi)$ along $\overline{\partial}\varphi$ enters. 
If $\varphi$ is holomorphic along $V$, this component is transverse to $V$. In order to estimate 
this component, we therefore only need families of functions with big Hessians in directions 
transverse to $V$. Such families should exist, the presence of $V$ notwithstanding. This is 
indeed the case, and the desired families can be obtained form $\log B(z,z)$, where $B$ is the 
Bergman kernel.

We begin with a lemma that makes this discussion rigorous. 
Denote by $(g_{jk})_{j,k=1,\ldots,n}(z)$ the Bergman metric at the point $z\in\Omega$, 
i.e. $g_{jk}=\partial^{2}\log B(z,z)/\partial z_{j}\partial\zb_k$. A good reference 
for properties of the Bergman metric is  \cite[Chapter VI]{JarnickiPflugBook2ndEd}. 
\begin{lemma}\label{LEMMA2}
Let $\Omega$ be a bounded convex domain, $0\leq q\leq (n-1)$, and 
$\varphi\in C^{1}(\overline{\Omega})$. Assume that $\varphi$ is holomorphic along 
varieties in $b\Omega$ of dimension $(q+1)$ (or higher). Then, for every $\varepsilon >0$, 
there exists a relatively compact subdomain $\Omega_{\varepsilon}$ of $\D$ with the 
following property. For each $z\in\Omega\setminus\Omega_{\varepsilon}$ and each set 
$t_{1}(z),\ldots, t_{q+1}(z)$ of orthonormal eigenvectors of $(g_{jk})(z)$, if the sum 
of the corresponding eigenvalues is $\leq \varepsilon^{-1}$, then 
$\left|(\overline{\partial}\varphi(z), t_{s}(z))\right| \leq \varepsilon$, $s=1,\ldots,(q+1)$.
\end{lemma}
We have identified $\overline{\partial}\varphi(z)$ with a vector in the obvious way, 
and $(\cdot,\cdot)$ denotes the inner product in $\mathbb{C}^{n}$.
\begin{proof} 
First of all, let $0\leq q\leq (n-2)$. We argue indirectly. Assume that the conclusion 
of the lemma is false. Then there exist a sequence $\{z_m\} \subset \D, z_0\in b\D$  
and orthonormal eigenvectors $t_1(z_m),\ldots,t_{q+1}(z_m)$ for the matrix 
$(g_{jk})(z_m)$ with associated eigenvalues $\sigma_{1}(z_{m}), \ldots,\sigma_{q+1}(z_{m})$ 
with the following three properties: 
\begin{itemize}
\item[(i)] $z_m\rightarrow z_0$ as $m\to \infty$,   
\item[(ii)] $\sigma_{1}(z_m)+\cdots+\sigma_{q+1}(z_m)\leq \varepsilon^{-1}$, 
\item[(iii)] for every $m$ there is $s_m$, $1\leq s_{m}\leq (q+1)$, 
with $|(\dbar\varphi(z_m),t_{s_{m}}(z_{m})|>\ep$. 
\end{itemize}
After passing to a suitable subsequence, we may assume that $s_m=s_{0}$ for all $m$ 
($1\leq s_{0}\leq(q+1)$) and that the frames $t_1(z_m),\ldots,t_{q+1}(z_m)$ converge 
to an orthonormal frame $t_{1}(z_{0}), \ldots, t_{q+1}(z_{0})$ at $z_0$ which spans 
a $(q+1)$--dimensional affine subspace of $\mathbb{C}^{n}$.

We next show that this affine subspace intersects $\overline{\Omega}$ in a nontrivial 
$(q+1)$--dimensional (affine) variety contained in $b\Omega$. In order to do this, 
we need a notion from \cite{Frankel91} that is clearly relevant for dealing with 
affine varieties in the boundary. For $z\in\Omega$ and $v\in\mathbb{C}^{n}$, set 
\begin{equation}\label{}
\delta_{\Omega}(z,v):= \min\{|w-z|:w\in b\Omega\cap(z+\mathbb{C}v)\}.
\end{equation}
On convex domains, the quantity $|v|/\delta_{\Omega}(z,v)$ (referred to in \cite{Frankel91} 
as the complex $1/d$--metric) satisfies a crucial comparison to the Bergman metric 
(\cite{Frankel91}, \cite[Theorem 4.3]{Zimmer21} and see \cite{NikPflugZwonek11} 
for this result on $\mathbb{C}$--convex domains): for 
all $n\in\mathbb{N}$, there is a constant $A_{n}$ such that if $\Omega$ is a bounded 
convex domain in $\mathbb{C}^{n}$, then
\begin{equation}\label{comparison}
 \frac{1}{A_{n}}\frac{|v|}{\delta_{\Omega}(z,v)} \leq \left((g_{jk})(z)(v,v)\right)^{1/2} 
\leq A_{n}\frac{|v|}{\delta_{\Omega}(z,v)}.
\end{equation}
The left inequality gives
\begin{equation}\label{lowerbound}
\delta_{\Omega}(z_{m},t_{s}(z_{m})) 
\geq \frac{\varepsilon^{1/2}}{A_{n}},\;s=1,\ldots,(q+1),\;m=1, 2, \ldots
\end{equation}
(since each $\sigma_{s}(z_{m}) \leq 1/\varepsilon$). Therefore for $1\leq s\leq (q+1)$, 
the affine discs $V_{s}:=\{z_{0}+\zeta t_{s}(z_{0}):|\zeta|\leq (\varepsilon^{1/2}/A_{n})\}$ 
are contained in $\overline{\Omega}$. As a convex domain, $\Omega$ admits a 
plurisubharmonic defining function (obtained from its Minkowski functional or 
gauge function, which is convex (\cite[page 35]{Rockafellar70}, 
\cite[Lemma 6.2.4]{ChenShawBook}) and therefore plurisubharmonic). The value of 
such a defining function at the center of $V_{s}$ (i.e. at $z_{0}$) is zero, while it is 
less than or equal to zero throughout the disc. The strong maximum principle for 
subharmonic functions implies that the value is zero throughout, that is 
$V_{s}\subset b\Omega$. Essentially the same argument shows that the convex 
hull of $\cup_{s=1}^{q+1}V_{s}$ is also contained in the boundary. By assumption, 
$\varphi$ is holomorphic along this variety. In particular, 
$(\overline{\partial}\varphi(z_{0}), t_{s_{0}}(z_0)) = 0$. This contradicts the 
fact that $|(\overline{\partial}\varphi(z_{m}), t_{s_{0}}(z_{m}))| > \varepsilon$.

If $q=(n-1)$, \eqref{comparison} shows that 
$\sigma_{1}(z_m)+\cdots+\sigma_{n}(z_m)\rightarrow\infty$ when 
$z_{m}\rightarrow b\Omega$, since $\delta_{\Omega}(z_{m},v)\rightarrow 0$ 
for $v$ transverse to $b\Omega$. So this sum will always be greater than 
$\varepsilon^{-1}$ for $z$ outside a relatively compact subdomain 
$\Omega_{\varepsilon}$, and the statement of the lemma holds.
\end{proof}

We are now ready for the proof of Theorem \ref{THEOREM1}.
\begin{proof}[Proof of Theorem \ref{THEOREM1}]
Denote by $\lambda(z)$ the Bergman potential; $\lambda(z):=\log B(z,z)$. The property 
that is crucial for us is that $\lambda$ has self bounded gradient on $\Omega$: 
\begin{equation}\label{selfbounded}
\left|\sum_{j=1}^{n}\frac{\partial\lambda}{\partial z_{j}}(z)w_{j}\right|^{2} 
\leq C\sum_{j,k=1}^{n}\frac{\partial^{2}\lambda}{\partial z_{j}\partial\zb_k}(z)
w_{j}\overline{w_{k}}\;;\;z\in\Omega,\;w\in\mathbb{C}^{n}.
\end{equation}
That is, $\partial\lambda$, measured in the metric induced by 
$\frac{\partial^{2}\lambda}{\partial z_{j}\partial\zb_k}$, is bounded 
(\cite[Proposition 4.6]{Zimmer21}).

In order to carry out the proof idea sketched above, we need a family of functions 
$\lambda_{\varepsilon}$ with uniformly (in $\varepsilon$) self bounded gradients, 
smooth in a neighborhood of the boundary, and with 
$\frac{\partial^{2}\lambda_{\varepsilon}}{\partial z_{j}\partial\zb_k}$ satisfying 
the properties from Lemma \ref{LEMMA2}, for the given $\varepsilon$.  Without loss of 
generality, assume that $0\in\Omega$. Let $\Omega_{\varepsilon}$ be from 
Lemma \ref{LEMMA2} and $0<\ep'\leq \ep$ small enough so that  
$\widetilde{\Omega}_{\varepsilon}:=(1+\varepsilon')\Omega_{\varepsilon}\Subset\Omega$. We set 
\begin{equation}\label{EQ50}
\lambda_{\ep}(z):=\lambda\left(\frac{z}{1+\ep'}\right).
\end{equation}
Then $\lambda_{\ep}\in C^{\infty}((1+\ep')\D)$ and  
\begin{align}\label{EQ51}
\partial\lambda_{\ep}(z)=\frac{1}{1+\ep'}\partial\lambda\left(\frac{z}{1+\ep'}\right); \quad 
\frac{\partial^2\lambda_{\ep}}{\partial z_j\partial \zb_k}(z) 
=\frac{1}{(1+\ep')^2}\frac{\partial^2\lambda}{\partial z_j\partial \zb_k} 
\left(\frac{z}{1+\ep'}\right). 
\end{align}
Because $\lambda$ has self bounded gradient, one easily checks that the 
$\lambda_{\varepsilon}$ have gradients that are self bounded uniformly in $\varepsilon$.

Let $z$ be in $\D\setminus \widetilde{\D}_{\ep}$ and assume that  the matrix 
$\left(\frac{\partial^2\lambda_{\ep}}{\partial z_j\partial \zb_k}(z)\right)$ 
has a set of $(q+1)$ eigenvalues $\sigma_{\ep,1}(z),\ldots, \sigma_{\ep,q+1}(z)$ with 
corresponding eigenvectors $t_{\ep,1}(z),\ldots,t_{\ep,q+1}(z)$  such that 
$\sigma_{\ep,1}(z)+\cdots+\sigma_{\ep,q+1}(z)\leq \ep^{-1}$. 
We claim that then
\begin{align}\label{EQ52}
|(\dbar\varphi(z),t_{\varepsilon,s}(z))|\leq \omega(\varepsilon)\;;\;1\leq s\leq (q+1), 
\end{align}
where $\omega(\varepsilon)$ denotes a quantity that is independent of $z$ and 
tends to zero as $\varepsilon$ tends to zero.
To see this, note that the eigenvalues of 
$\left(\frac{\partial^{2}\lambda_{\varepsilon}}{\partial z_{j}\partial\zb_k}\right)(z)$ 
are $1/(1+\varepsilon')^{2}$ times those of 
$\left(\frac{\partial^{2}\lambda}{\partial z_{j}\partial\zb_k}\right)(z/(1+\varepsilon'))$ 
(see \eqref{EQ51}), and the associated eigenvectors are the same. Because 
$\varphi\in C^{1}(\overline{\Omega})$, we also have that 
$|\overline{\partial}\varphi(z) - \overline{\partial}\varphi(z/(1+\varepsilon'))|$ tends to 
zero uniformly for $z\in\overline{\Omega}$ as $\varepsilon$ tends to zero. Combining 
these two observations with a triangle inequality argument yields \eqref{EQ52}; 
the resulting quantity $\omega(\varepsilon)$ involves, among other things, the modulus 
of continuity of $\overline{\partial}\varphi$ on $\overline{\Omega}$.

We are now in a position to establish \eqref{compest}. Fix $\ep>0$ and pick $\Omega_{\ep}$ 
and $\lambda_{\varepsilon}$ from Lemma \ref{LEMMA2} and \eqref{EQ50}, respectively. 
To simplify notation, set
\begin{align}\label{EQ32}
v:=N_{q+1}(\dbar\varphi\wedge f).
\end{align}
Then $ v\in \ker(\dbar)\cap \text{dom}(\dbar^*)$ and $\|v\|+\|\dbar^*v\|\lesssim \|f\|$. 
As in the proof of Theorem 4.29 in \cite{StraubeBook} or Theorem 2.1 in \cite{StraubeICM} 
(that property $(\widetilde{P_{q}})$ implies compactness of $N_{q}$), we need an estimate 
that brings the Hessian of $\lambda_{\varepsilon}$ into play. If $(\lambda_{\varepsilon})$ 
is a family of functions with uniformly self bounded gradient, and 
$u\in \ker(\overline{\partial})\cap \text{dom}(\overline{\partial}^{*})\subseteq L^{2}_{(0,q)}(\Omega)$, 
then
\begin{equation}\label{crucialestimate}
 \int_{\Omega}\sumprime_{|K|=q}\sum_{j,k}
\frac{\partial^{2}\lambda_{\varepsilon}}{\partial z_{j}
\partial\zb_k}(e^{-\lambda_{\varepsilon}/2}u)_{jK}
\overline{(e^{-\lambda_{\varepsilon}/2}u)_{kK}}\,dV 
\lesssim \|\overline{\partial}^{*}(e^{-\lambda_{\varepsilon}/2}u)\|^{2},
\end{equation}
with a constant that does not depend on $\varepsilon$. This estimate results from 
the Morrey--Kohn--H\"{o}rmander formula with weight $\lambda_{\varepsilon}$ 
(see, for instance, \cite[Proposition 2.4]{StraubeBook}). 
On the left hand side, we have distributed the weight factor $e^{-\lambda_{\varepsilon}}$ 
into the form $e^{-\lambda_{\varepsilon}/2}u$. There is no $\dbar$--term on 
the right hand side, as $\dbar u=0$, and replacing the (weighted) 
$\dbar^*_{\lambda_{\varepsilon}}$--term in \cite[(2.24)]{StraubeBook} 
by the right hand side of \eqref{crucialestimate} makes an error of the order 
$|\partial\lambda_{\varepsilon}(e^{-\lambda_{\varepsilon}/2}u)|$. 
This error can be absorbed into the left hand side by the self bounded 
gradient condition. This assumes that the constant in the self bounded 
gradient estimate is small enough. This property can always be 
achieved by scaling $\lambda_{\varepsilon}$ to $c\lambda_{\varepsilon}$ 
if necessary (the two sides of \eqref{selfbounded} scale with $c^{2}$ 
and $c$, respectively). This scaling does not affect the argument, 
and we continue with $\lambda_{\varepsilon}$. Details are in \cite{StraubeBook}, 
estimate (4.80). There is a compactly supported term on the right hand side in (4.80) 
in \cite{StraubeBook}, but this term is the result of the function $\lambda_{M}$ 
there satisfying the self bounded gradient assumption only in a neighborhood 
of the boundary, and so does not occur here. 

The factor $e^{-\lambda_{\varepsilon}/2}$ with the form $u$ in \eqref{crucialestimate} 
appears problematic at first, as there is no such factor in the right hand side of 
\eqref{norm2} (the quantity we want to estimate). However, the following observation 
from \cite{StraubeICM, StraubeBook} allows to handle this difficulty. If $v$ is 
$\overline{\partial}$--closed and we define 
\begin{align}\label{EQ32-2}
u_{\ep}:=P_{q+1,\lambda_{\ep}/2}(e^{\lambda_{\ep}/2}v) 
\in\ker(\dbar)\cap\text{dom}(\dbar^*),
\end{align} 
then one has
\begin{align}\label{EQ33}
P_{q+1}(e^{-\lambda_{\ep}/2}u_{\ep})=v ;\text{ and then also }
\overline{\partial}^{*}(e^{-\lambda_{\varepsilon}/2}u_{\ep}) = \overline{\partial}^{*}v.
\end{align}
Here, $P_{q+1}$ and $P_{q+1,\lambda_{\ep}/2}$ are the Bergman projection and 
the weighted Bergman projection on $(0,q+1)$-forms on $\Omega$, respectively. 
One can check \eqref{EQ33} by pairing with $\overline{\partial}$--closed forms, 
see \cite[(4.83) -- (4.85)]{StraubeBook}. The point of \eqref{EQ33} in the present 
context is that because $\dbar\varphi\wedge f$ is $\dbar$-closed, the quantity 
we need to estimate now becomes
\begin{align}\label{EQ34}
\langle v,\dbar\varphi\wedge f\rangle
=\langle e^{-\lambda_{\ep}/2}u_{\ep},\dbar\varphi\wedge f\rangle,
\end{align}
and we have the exponential factor. Moreover, by \eqref{EQ33}, the right hand 
side of \eqref{crucialestimate} equals 
$\overline{\partial}^{*}v = \overline{\partial}^{*}N_{q+1}(\overline{\partial}\varphi\wedge f)$, 
and so is dominated by $\|f\|$ (independently of $\varepsilon$). Therefore, we will  
get good estimates on $e^{-\lambda_{\varepsilon}/2}u_{\varepsilon}$, hence on 
$\langle v,\overline{\partial}\varphi\wedge f\rangle$, in directions where 
$\lambda_{\varepsilon}$ has big Hessian. In the other directions, we will estimate 
$\langle v,\overline{\partial}\varphi\wedge f\rangle$ via \eqref{EQ52}.

The argument proceeds as follows. Let $z\in \D\setminus \widetilde{\D}_{\ep}$ 
(fixed for now). We want to estimate 
\begin{equation}\label{}
\left( e^{-\lambda_{\varepsilon}/2}u_{\varepsilon},\overline{\partial}\varphi\wedge f\right)(z) 
=\sumprime_{|J|=q+1}\left((e^{-\lambda_{\varepsilon}/2}u_{\varepsilon})_{J},
(\overline{\partial}\varphi\wedge f)_{J}\right)(z).
\end{equation}
To do this, we work relative to a basis $\{t_{\ep,1}(z),\ldots, t_{\ep,n}(z)\}$ of 
eigenvectors of $(\frac{\partial^{2}\lambda_{\varepsilon}}{\partial z_{j}\partial\zb_k})(z)$ 
with associated eigenvalues $\sigma_{\varepsilon,1}(z), \ldots, \sigma_{\varepsilon,n}(z)$.  
Fix a multi index $J = (j_{1},\ldots, j_{q+1})$. Then
\begin{multline}\label{EQ70}
\sumprime_{|K|=q}\sum_{j,k}
\frac{\partial^{2}\lambda_{\varepsilon}}{\partial z_{j}\partial\zb_k}(z) 
(e^{-\lambda_{\varepsilon}/2}u_{\varepsilon})_{jK}
\overline{(e^{-\lambda_{\varepsilon}/2}u_{\varepsilon})_{kK}} \\
= \sumprime_{|K|=q}\sum_{j}\sigma_{\varepsilon,j}(z) |(e^{-\lambda_{\varepsilon}/2}u_{\varepsilon})_{jK}|^{2} 
\geq (\sigma_{\ep,j_{1}}(z)+ \cdots +\sigma_{\ep,j_{q+1}}(z))|(e^{-\lambda_{\varepsilon}/2}u_{\varepsilon})_{J}|^{2}\;;
\end{multline}
compare equation (4.30) in \cite{StraubeBook}. The reason for the last inequality is
 that $(u_{\varepsilon})_{J}$ appears via $(q+1)$ pairs $(j,K)$, namely for $j=j_1,\ldots, j=j_{q+1}$; 
the other terms not involving $J$ are nonnegative. If 
$\sigma_{\ep,j_{1}}(z)+\cdots +\sigma_{\ep,j_{q+1}}(z) \geq 1/\varepsilon$,
we obtain 
\begin{align}\label{EQ38}
|(e^{-\lambda_{\ep}/2}u_{\ep})_{J}(z)|^2
&\leq\ep\sumprime_{|K|=q}\sum\limits_{j,k}
\frac{\partial^2\lambda_{\ep}}{\partial z_{j}\partial\zb_{k}}(z) 
 (e^{-\lambda_{\ep}/2}u_{\ep})_{jK}\overline{(e^{-\lambda_{\ep}/2}u_{\ep})}_{kK}.
\end{align}
If $\sigma_{\ep,j_{1}}(z)+\cdots +\sigma_{\ep,j_{q+1}} (z)< 1/\varepsilon$, then 
$|(\dbar\varphi(z),t_{\ep,j_{s}}(z))|\leq \omega(\varepsilon)$ for $1\leq s\leq q+1$ 
(see \eqref{EQ52}; $z\in \D\setminus \widetilde{\D}_{\ep}$). Therefore
\begin{align}\nonumber 
|(\dbar\varphi\wedge f)_J(z)|=&|(\dbar\varphi\wedge f)(z)(t_{\ep,j_{1}},t_{\ep,j_{2}},\ldots,t_{\ep,j_{q+1}})|\\
\label{EQ42}\leq &\sum\limits_{s=1}^{q+1} |(\dbar\varphi,t_{\ep,j_{s}})(z)| 
|f(t_{\ep,j_{1}},t_{\ep,j_{2}},\ldots, \check{t_{\ep,j_{s}}},\ldots,t_{\ep,j_{q+1}})|
 \lesssim\omega(\varepsilon)|f(z)|.
\end{align}
To control $e^{-\lambda_{\varepsilon}/2}u_{\varepsilon}$ in this case, we use that 
$(\frac{\partial^{2}\lambda}{\partial z_{j}\partial\zb_k})$ is uniformly bounded from 
below; this can be seen for example from \eqref{comparison}, because $\Omega$ is 
bounded. Therefore
\begin{align}\label{EQ39}
|e^{-\lambda_{\ep}/2}u_{\ep}(z)|^2
&=\frac{1}{q+1}\sumprime_{|K|=q}\sum\limits_{j} |e^{-\lambda_{\ep}/2}(u_{\ep})_{jK}(z)|^2\\
\nonumber&\lesssim \frac{1}{q+1}\sumprime_{|K|=q}\sum\limits_{j,k} 
\frac{\partial^2\lambda_{\ep}}{\partial z_{j}\partial\zb_{k}} (z)
(e^{-\lambda_{\varepsilon}/2}u_{\ep})_{jK}\overline{(e^{-\lambda_{\varepsilon}/2}u_{\ep})_{kK}}\;;
\end{align}
in the last step $\lesssim$ is independent of $f$ and $\ep$. Thus we obtain from both cases 
that when $z\in\D\setminus \widetilde{\D}_{\ep}$, 
\begin{align*} 
\left|\left( e^{-\lambda_{\ep}/2}u_{\ep},\dbar\varphi\wedge f\right)(z)\right|^2
\lesssim (\ep+\omega^2(\varepsilon))\left(\sumprime_{|K|=q}\sum\limits_{j,k} 
\frac{\partial^2\lambda_{\ep}}{\partial z_{j}\partial\zb_{k}}(z)(e^{-\lambda_{\varepsilon}/2}u_{\ep})_{jK} 
\overline{(e^{-\lambda_{\varepsilon}/2}u_{\ep})_{kK}}\right)|f(z)|^2.
\end{align*}
In the estimate above $\lesssim$ is independent of $\ep$ and $f$. Setting 
$\widetilde{\omega}(\varepsilon):= \sqrt{\varepsilon + \omega^2(\varepsilon)}$, we have in view 
of \eqref{crucialestimate}
\begin{align}
\nonumber\left|\langle e^{-\lambda_{\ep}/2}u_{\ep}, 
\dbar\varphi\wedge f\rangle_{\Omega\setminus\widetilde{\Omega}_{\varepsilon}}\right| 
&\lesssim \widetilde{\omega}(\varepsilon)\left(\int\limits_{\D}\sumprime_{|K|=q}\sum\limits_{j,k} 
\frac{\partial^2\lambda_{\ep}}{\partial z_{j}\partial\zb_{k}}
(e_{\lambda_{\varepsilon}/2}u_{\ep})_{jK}
\overline{(e^{-\lambda_{\varepsilon}/2}u_{\ep})}_{kK}\,dV\right)^{1/2} 
 \left(\int\limits_{\D}|f|^2dV\right)^{1/2}\\
\label{EQN40}&\lesssim\widetilde{\omega}(\varepsilon)\left(\int\limits_{\D}\sumprime_{|K|=q}\sum\limits_{j,k} 
\frac{\partial^2\lambda_{\ep}}{\partial z_{j}\partial\zb_{k}} (e^{-\lambda_{\varepsilon}/2}u_{\ep})_{jK}
\overline{(e^{-\lambda_{\varepsilon}/2}u_{\ep})}_{kK}\,dV +\int\limits_{\D}|f|^2dV\right)\\
\nonumber&\lesssim\widetilde{\omega}(\varepsilon)\left(\|\dbar^*(e^{-\lambda_{\ep}/2}u_{\ep})\|_{\D}^2  
+\|f\|_{\D}^2\right)\\
\nonumber &\lesssim \widetilde{\omega}(\varepsilon) \left(\|\dbar^* v\|^2_{\D}+\|f\|_{\D}^2\right) 
\lesssim \widetilde{\omega}(\varepsilon)\|f\|_{\D}^2\;. 
\end{align}
We have used here that the integrands are nonnegative, so that the integrals over $\Omega$ 
dominate those over $\Omega\setminus\widetilde{\Omega}_{\varepsilon}$; the second inequality 
is $|ab|\leq a^{2}+b^{2}$.

It remains to control 
$\langle e^{-\lambda_{\ep}/2}u_{\ep}, \dbar\varphi\wedge f\rangle_{\widetilde{\Omega}_{\varepsilon}}$. 
To this end, consider the following three linear operators: 
\vspace{-0.28in}
\begin{align*}
K_{1,\varepsilon}&: K^{2}_{(0,q)}(\Omega)\rightarrow \text{dom}(\overline{\partial})
\cap \text{dom}(\overline{\partial})^{*}\subset L^{2}_{(0,q)}(\Omega),\\
K_{2,\varepsilon}&: \text{dom}(\overline{\partial})\cap \text{dom}(\overline{\partial}^{*})
\subset L^{2}_{(0,q)}(\Omega)\rightarrow \text{dom}(\overline{\partial})\cap \text{dom}(\overline{\partial}^{*})
\subset L^{2}_{(0,q)}(\Omega),\\
K_{3,\varepsilon}&: \text{dom}(\overline{\partial})\cap \text{dom}(\overline{\partial}^{*})
\subset L^{2}_{(0,q)}(\Omega)\rightarrow L^{2}_{(0,q)}(\widetilde{\Omega}_{\varepsilon}), 
\end{align*}
defined as follows: 
$K_{1,\ep}f=N_{q}(\overline{\partial}\varphi\wedge f)$, 
$K_{2,\ep}v=e^{-\lambda_{\varepsilon}/2}(P_{q,\lambda_{\varepsilon}/2})(e^{\lambda_{\varepsilon}/2}v)$, 
and $K_{3,\ep}w=w|_{\widetilde{\Omega}_{\varepsilon}}$. Then $K_{1,\varepsilon}$ and 
$K_{2,\varepsilon}$ are continuous, and $K_{3,\varepsilon}$ is compact (by interior elliptic 
regularity, see for example Proposition 5.1.1 in \cite{ChenShawBook}, and the fact 
$W^{1}(\widetilde{\Omega}_{\varepsilon})\hookrightarrow L^{2}(\widetilde{\Omega}_{\varepsilon})$ 
is compact). Consequently, the operator 
$K_{\varepsilon} = K_{3,\varepsilon}\circ K_{2,\varepsilon}\circ K_{1,\varepsilon}$ is compact. We have
\begin{equation}\label{EQ53}
 \left|\langle e^{-\lambda_{\ep}/2}u_{\ep}, 
\dbar\varphi\wedge f\rangle_{\widetilde{\Omega}_{\varepsilon}}\right| 
= \left|\langle K_{\varepsilon}f, 
\overline{\partial}\varphi\wedge  f\rangle_{\widetilde{\Omega}_{\varepsilon}}\right| 
\lesssim \|K_{\varepsilon}f\|_{\widetilde{\Omega}_{\varepsilon}}\|f\|_{\Omega} 
\lesssim \|\varepsilon^{-1}K_{\varepsilon}f\|_{\widetilde{\Omega}_{\varepsilon}}^{2} 
+ \varepsilon^{2}\|f\|_{\Omega}^{2}\;;
\end{equation}
with the constants in the two inequalities being independent of $\varepsilon$ (and $f$). 
The second inequality is the usual small constant--large constant estimate.

To complete the proof of Theorem \ref{THEOREM1}, it now suffices to combine \eqref{norm2}, 
\eqref{EQ32}, \eqref{EQ34}, \eqref{EQN40}, and \eqref{EQ53} to obtain
\begin{equation}\label{final}
\|H^{q}_{\varphi}f\|^{2} \leq C\left((\widetilde{\omega}(\varepsilon) + \varepsilon^{2})\|f\|^{2} 
+ \|\varepsilon^{-1}K_{\varepsilon}f\|^{2}\right),
\end{equation}
where $C$ is independent of $\varepsilon$. Because $\varepsilon^{-1}K_{\varepsilon}$ is compact, 
and $\lim_{\varepsilon\rightarrow 0^{+}}(\widetilde{\omega}(\varepsilon)+\varepsilon^{2}) = 0$, 
we can rescale to obtain the required family of estimates \eqref{compest}.
\end{proof}


\section{Proof of Corollary \ref{Fredholm}}\label{FredholmProof}

The properties of Fredholm operators we use can be found for example in \cite[Chapter XI]{ConwayBook}. 
We first note that we may assume that $\varphi\neq 0$ on $\overline{\Omega}$ also when $q=0$. 
Indeed, if $\varphi\neq 0$ on $b\Omega$, then there is $\widetilde{\varphi}\in C^{1}(\overline{\Omega})$ 
with $\widetilde{\varphi}\neq 0$ on $\overline{\Omega}$ and $\varphi - \widetilde{\varphi}$ compactly 
supported. Then $T^{0}_{\varphi} = T^{0}_{\widetilde{\varphi}} + T^{0}_{\varphi-\widetilde{\varphi}}$.
Since $f\mapsto(\varphi-\widetilde{\varphi})f$ is compact when $q=0$ the operator 
$T^{0}_{\varphi-\widetilde{\varphi}}$ is compact. Hence $T^{0}_{\varphi}$ and 
$T^{0}_{\widetilde{\varphi}}$ are simultaneously Fredholm or non-Fredholm, 
with equal indices in the first case.

So assume now that $\varphi\in C^{1}(\overline{\Omega})$ is nonvanishing on $\overline{\Omega}$, 
and $0\leq q\leq (n-1)$. Then $(1/\varphi)\in C^{1}(\overline{\Omega})$ and $(1/\varphi)$ is also 
holomorphic along varieties in the boundary. We have
\begin{equation}\label{Calkin1}
T^{q}_{\varphi}T^{q}_{(1/\varphi)} 
= T^{q}_{\varphi(1/\varphi)} - (H^{q}_{\overline{\varphi}})^{*}H^{q}_{(1/\varphi)}\;,
\end{equation}
and
\begin{equation}\label{Calkin2}
T^{q}_{(1/\varphi)}T^{q}_{\varphi} 
= T^{q}_{(1/\varphi)\varphi} - (H^{q}_{(\overline{1/\varphi})})^{*}H^{q}_{\varphi}\;,
\end{equation}
where $^{*}$ denotes the adjoint. These equations are the standard relations between (semi) 
commutators of Toeplitz operators and Hankel operators; they follow by direct computation. 
Of course, $T^{q}_{\varphi(1/\varphi)}=T^{q}_{(1/\varphi)\varphi}=\mathbb{I}$, the identity. 
By Theorem \ref{THEOREM1}, both $H^{q}_{(1/\varphi)}$ and $H^{q}_{\varphi}$ are compact. 
So \eqref{Calkin1} and \eqref{Calkin2} say that $T^{q}_{\varphi}$ is invertible modulo compact 
operators (with inverse in the Calkin algebra given by $T^{q}_{(1/\varphi)}$), that is , 
$T^{q}_{\varphi}$ is Fredholm.

It remains to see that $\ind(T^{q}_{\varphi})=0$. Because $\overline{\Omega}$ is simply connected, 
there is $\widehat{\varphi}\in C(\overline{\Omega})$ such that 
$\varphi = e^{\widehat{\varphi}}$ ($z\mapsto e^{z}$ is a covering map of $\mathbb{C}$ 
onto $\mathbb{C}\setminus \{0\}$). Moreover, $\widehat{\varphi}\in C^{1}(\overline{\Omega})$, 
and it is holomorphic along varieties in $b\Omega$. This is clear for a local branch of $\log\varphi$, 
and such a branch differs from $\widehat{\varphi}$ by a constant. We can now apply a standard 
homotopy argument. Namely, for all $t\in [0,1]$, $T^{q}_{e^{t\widehat{\varphi}}}$ is Fredholm 
(by Theorem \ref{THEOREM1}). Also, $t\mapsto T^{q}_{e^{t\widehat{\varphi}}}$ is continuous 
from $[0,1]$ to $\mathcal{B}(K^{2}_{(0,q)}(\Omega))$, the bounded operators on 
$K^{2}_{(0,q)}(\Omega)$ ($\varphi\mapsto T^{q}_{\varphi}$ is continuous form $C(\overline{\Omega})$ 
to $\mathcal{B}(K^{2}_{(0,q)}(\Omega)$). The index is continuous from the Fredholm operators 
with the topology inherited from $\mathcal{B}(K^{2}_{(0,q)}(\Omega))$ to $\mathbb{Z}$ 
(with the discrete topology). Thus $t\mapsto \ind(T^{q}_{e^{t\widehat{\varphi}}})$ is 
continuous from $[0,1]$ to $\mathbb{Z}$, and therefore is constant. We obtain 
$0 = \ind(T^{q}_{e^{t\widehat{\varphi}}})|_{t=0} = \ind(T^{q}_{e^{t\widehat{\varphi}}})|_{t=1} 
=\ind(T^{q}_{\varphi})$, 
and the proof of Corollary \ref{Fredholm} is complete.


\end{document}